\documentclass{amsart}

\usepackage{indentfirst}
\usepackage{amssymb}
\usepackage{amsmath}
\usepackage{graphics}
\usepackage{epsfig}
\usepackage[usenames]{color}
\usepackage{MnSymbol,wasysym}

\usepackage[utf8]{inputenc}
\usepackage{latexsym}
\usepackage{graphicx}

\usepackage[normalem]{ulem}

\usepackage{stmaryrd}
\usepackage{hyperref}

\usepackage[T1]{fontenc}
\usepackage{kantlipsum}
\usepackage[usenames,dvipsnames]{xcolor}
\usepackage[breakable, theorems, skins]{tcolorbox}
\tcbset{enhanced}

\definecolor{light-gray}{gray}{0.95}

\catcode`@=11 \@addtoreset{equation}{section}
\renewcommand\theequation{\thesection.\@arabic\c@equation}
\catcode`@=12

\newcommand{\RR}{\mathbb{R}}

\expandafter\chardef\csname pre amssym.def
at\endcsname=\the\catcode`\@ \catcode`\@=11
\def\undefine#1{\let#1\undefined}
\def\newsymbol#1#2#3#4#5{\let\next@\relax
 \ifnum#2=\@ne\let\next@\msafam@\else
 \ifnum#2=\tw@\let\next@\msbfam@\fi\fi
 \mathchardef#1="#3\next@#4#5}
\def\mathhexbox@#1#2#3{\relax
 \ifmmode\mathpalette{}{\m@th\mathchar"#1#2#3}%
 \else\leavevmode\hbox{$\m@th\mathchar"#1#2#3$}\fi}
\def\hexnumber@#1{\ifcase#1 0\or 1\or 2\or 3\or 4\or 5\or 6\or 7\or 8\or
 9\or A\or B\or C\or D\or E\or F\fi}

\font\teneufm=eufm10 \font\seveneufm=eufm7 \font\fiveeufm=eufm5
\newfam\eufmfam
\textfont\eufmfam=\teneufm \scriptfont\eufmfam=\seveneufm
\scriptscriptfont\eufmfam=\fiveeufm

\catcode`\@=\csname pre amssym.def at\endcsname

\newcommand{\eqn}{\begin{eqnarray}}
\newcommand{\een}{\end{eqnarray}}

\newtheorem {Theorem}  {Theorem}

\numberwithin{Theorem}{section}
\newtheorem{Lemma}[Theorem]{Lemma}

\newtheorem{Definition}[Theorem]{Definition}
\newtheorem{Remark}[Theorem]{Remark}

\newcommand{\N}{{\mathbb N}}

\DeclareMathOperator{\sgn}{sgn}

\newcommand{\R}{\mathbf{R}}

\begin{document}

\title[ON A FAMILY OF REGULARIZED BENJAMIN-TYPE EQUATIONS]{Global well-posedness for a family of regularized Benjamin-type equations}

\author[I.P. Bastos]{Izabela Patr\'{\i}cio Bastos}
\address[I.P. Bastos]{Escola Politécnica, PUCPR, CEP 80215-901, Curitiba - PR, Brazil}
\email{izabela.patricio@pucpr.br}

\author[D.G. Alfaro Vigo]{Daniel G. {Alfaro Vigo}}
\address[D.G. Alfaro Vigo]{Departamento de Computa\c{c}\~ao, Instituto de Computa\c{c}\~ao, Universidade Federal do Rio de Janeiro, CP 68530, CEP 21941-590, Rio de Janeiro - RJ, Brazil}
\email{dgalfaro@ic.ufrj.br}

\author[A.Ruiz de Zarate]{Ailin Ruiz de Zarate Fabregas}
\address[A. Ruiz de Zarate]{Departamento de Matem\'atica, Universidade Federal do Paran\'a, CP 19081, CEP 81531-980, Curitiba - PR, Brazil}
\email{ailin@ufpr.br}

\author[J.  Schoeffel]{Janaina Schoeffel}
\address[J.  Schoeffel]{Setor de Educa\c{c}\~ao Profissional e Tecnol\'ogica, Universidade Federal do Paran\'a, CEP 81520-260, Curitiba - PR, Brazil}
\email{janainaschoeffel@ufpr.br}

\author[C. J. Niche]{C\'esar J. Niche}
\address[C.J. Niche]{Departamento de Matem\'atica Aplicada, Instituto de Matem\'atica. Universidade Federal do Rio de Janeiro, CEP 21941-909, Rio de Janeiro - RJ, Brazil}
\email{cniche@im.ufrj.br}

\thanks{I. P. Bastos was partially supported by Coordena\c{c}\~ao de Aperfei\c{c}oamento de Pessoal de N\'{\i}vel Superior - Brazil (CAPES) - Finance Code 001.  C.J.  Niche was partially supported by PROEX CAPES  and Bolsa PQ CNPq - 308279/2018-2, Brazil.}



\date{\today}

\begin{abstract}
In this work we prove local and global well-posedness results for the Cauchy problem of a family of regularized nonlinear Benjamin-type equations in both periodic and nonperiodic Sobolev spaces.
\end{abstract}

\maketitle

\begin{center}
{\it In memoriam Rafael Jos\'e I\'orio J\'unior.}
\end{center}

\section{Introduction}\label{intro}

The purpose of this work is to prove local and global well-posedness of the Cauchy problem associated to a family of regularized nonlinear Benjamin-type equations given by
\begin{equation}\label{r-benjamin-family}
\eta_t + \eta_x - \frac{3}{2} \alpha \eta \eta_x -a \eta_{xxt} - b \mathcal{L}(\eta_{xt}) = 0,
\end{equation}
where $t$ and $x$ are nondimensional temporal and horizontal space variables, respectively and $\alpha, a, b$ are positive constants. The function $\eta = \eta(x,t)$ accounts for the displacement of the interface between two fluid layers of different densities. The operator $\mathcal{L}$ can be either $\mathcal{L}= \mathcal{H}$,  the {\it Hilbert transform},  defined in the frequency domain through 
\begin{displaymath}
\widehat{\mathcal{H}f}(k) =  \mathrm{i} \sgn{(k)} \widehat{f}(k),\ \ \ \ \   k \in \mathbb{R} \ (\text{or}\ \mathbb{Z}),\ k \neq 0, 
\end{displaymath}
or $\mathcal{L}= \mathcal{T}$, the {\it Hilbert transform on the strip of height $h>0$}, which is defined through
\begin{displaymath}
\widehat{\mathcal{T}f}(k) = \mathrm{i} \coth(hk) \widehat{f}(k),\ \ \ \ \   k \in \mathbb{R} \ (\text{or}\ \mathbb{Z}),\ k \neq 0.
\end{displaymath}
Here $\widehat{g}$ denotes the nonperiodic ($k \in \mathbb{R}$) or periodic ($k \in \mathbb{Z}$) Fourier transform of $g$ with respect to $x$.

Equation~(\ref{r-benjamin-family}) describes the evolution of unidirectional internal waves at the interface of two immiscible, inviscid, incompressible and irrotational fluids limited by a rigid lid at the top, so there is no surface wave. It was obtained as a unidirectional reduction of a Boussinesq-type system in~\cite{MR4505862}, when the lower fluid layer has finite depth of value $h>0$ at rest and the bottom is flat. In such configuration we have $\mathcal{L}= \mathcal{T}$, while as $h$ tends to infinity, so the lower fluid layer has infinite depth, $\mathcal{L}= \mathcal{H}$. The third order term $a \eta_{xxt}$ acts as a regularization term, as it is the case in the Benjamin–Bona–Mahony~(BBM) equation~\cite{MR427868} if compared to KdV equation. In this regard, equation~(\ref{r-benjamin-family}) can be associated to the Benjamin equation
\begin{displaymath}
u_t+u_x+2uu_x-\alpha \mathcal{H}[u]_{xx}-\beta u_{xxx}=0,\qquad \alpha,\beta>0,
\end{displaymath}
introduced in~\cite{MR1194985} to model internal solitary waves when the interface is subjected to capillarity effects and the lower layer has infinite depth. The existence, stability and asymptotic behavior of solitary and periodic waves of permanent form for the Benjamin equation, among other properties, were studied in~\cite{MR1194985, MR1400219, MR1608077, MR1726196, MR1672020,MR2149522,MR3213486,MR4300934}. For equation (\ref{r-benjamin-family}) in the case $\mathcal{L}= \mathcal{T}$, the proof of the existence of periodic travelling wave solutions was provided in~\cite{MR4505862}.  

Concerning well-posedness results, Linares proved in~\cite{MR1674557} global well-posedness in $L^2$ for the initial value problem associated to the equation   
\begin{displaymath}
u_t+2uu_x-\alpha \mathcal{H}[u]_{xx}+u_{xxx}=0,
\end{displaymath}
for both nonperiodic and periodic cases. He also first proved local well-posedness in Sobolev spaces of $L^2$-type, denoted as $H^s(\mathbb{R})$, for $s\geq 0$. These results were improved by Kozono et al.~\cite{MR1858761}, who proved local well-posedness in $H^s(\mathbb{R})$, for $s> -3/4$; Li and Wu~\cite{MR2661345}, who proved global well-posedness for the Benjamin equation in $H^s(\mathbb{R})$, for $s > -3/4$ and Chen et al.~\cite{MR2833406}, who extended the global well-posedness in $H^s(\mathbb{R})$, for $s\geq -3/4$, as pointed out in~\cite{MR3003295} and~\cite{MR4569255}. These last two articles deal with the initial value problem associated to the Benjamin equation in weighted Sobolev spaces.

To the extent of our knowledge there are no known published analytical results concerning the local or global well-posedness for the full family of models described by equation~(\ref{r-benjamin-family}) in Sobolev spaces $H^{s}$. Pivotal results for the problem in question were developed
 in~\cite{Ref_Izabela_2019}, which are now presented here in a more compact and general way. Therefore, the main goal in this work is to provide a detailed proof of our main result which reads as follows:
\begin{Theorem}
Let $s \geq 0$ and $\phi \in H^{s}$. Then,  the nonlinear Cauchy problem associated to equation~(\ref{r-benjamin-family}), this is, 
\begin{displaymath}
\left\{
\begin{aligned}    
& \eta_t + \eta_x - \frac{3}{2} \alpha \eta \eta_x -a \eta_{xxt} - b \mathcal{L}(\eta_{xt}) = 0,\\
&\eta(0)=\phi \ \in H^s,
\end{aligned}
\right.
\end{displaymath}
is globally well-posed for $\mathcal{L}= \mathcal{H}$ and $\mathcal{L}= \mathcal{T}$, both in the nonperiodic and periodic settings.
\end{Theorem}

To prove this, we first obtain the local well-posedness result through a standard fixed point argument. Then,  to show global well-posedness for $s \geq 1$,  we use a version of the extension principle through an appropiate conserved quantity equivalent to the $H^1$-norm.  In order to establish the result for $0 \leq s < 1$,  we decompose the initial datum in low and high frequency components and prove well-posedness of a Cauchy problem related to the evolution and interaction of these components. The aforementioned decomposition follows the strategy used in~\cite{MR2461849} in the case of the BBM equation and introduced in~\cite{MR1466164}.

When no third order term is present in equation~(\ref{r-benjamin-family}), i.e. $a = 0$, we are in the presence of the regularized intermediate long-wave~(rILW) equation in the finite depth case $\mathcal{L}= \mathcal{T}$, for which global well-posedness in $H^s$, $s>1/2$, was proven in~\cite{MR3805030},  while a deduction for a model in which the bottom is not assumed to be flat,  i.e.  there is non-trivial topography, was given in~\cite{MR2433701}.  When considering infinite depth, this is $\mathcal{L}= \mathcal{H}$, and again $a=0$, we are facing the regularized Benjamin-Ono~(rBO) equation and many analytical results describing its properties are known, see~\cite{MR1135490,MR2776905, MR3359523} and references therein. In general, the regularized equations mentioned here~(BBM, rILW, rBO), provide a unidirectional wave model of the same order of approximation than the corresponding equation~(KdV, BO, ILW), with respect to the asymptotic expansions in terms of powers of the dispersion and nonlinear parameter. These parameters appear in the bidirectional Boussinesq-type systems obtained as asymptotic reductions of the Euler equations. For such systems, there are also regularized versions of the same order of approximation, which are more tractable theoretically and numerically. All equations mentioned here are unidirectional reductions of bidirectional Boussinesq-type systems that were restricted to one propagation direction and the same advantage for the regularized equations were confirmed in~\cite{MR1135490}. That is why regularized bidirectional Boussinesq-type systems and regularized unidirectional equations appear so often in the recent literature, see for instance~\cite{MR2433701,MR2504630,MR2887985,MR3219727,MR3463169,
MR3853908,MR4414929} for internal wave models and~\cite{MR1969681,MR2128348,MR2057134,MR2251190,2007EPJST.147..113N,MR2365630,MR2921604,MR4216091} for surface wave models. For a systematic derivation of a wide class of asymptotic bidirectional models for internal waves see~\cite{Ref_BonaLannesSaut_2008,Ref_CraigGuyenneKalisch_2005}.
As in the examples mentioned before, the regularized terms in equation~(\ref{r-benjamin-family}) provides an advantage for numerical methods for the computation of approximate solutions, as was also developed in~\cite{Ref_Izabela_2019}.

This article is organized as follows. In Section~\ref{preliminary-linearproblem}  we gather definitions, set notation and prove global existence of solutions to the linearization of equation~(\ref{r-benjamin-family}). In Section~\ref{Localwellposedness} we prove local well-posedness for the Cauchy problem associated to equation~(\ref{r-benjamin-family}), while in Section~\ref{Globalwellposedness} we prove global well-posedness.

\section{Settings and linear theory results}
\label{preliminary-linearproblem}

In this section we recall some definitions used throughout this article and prove existence and uniqueness of the solution to the Cauchy problem of the linearized equation using results from the theory of semigroups of linear operators.

\subsection{Definitions and notation} 

The Fourier transform of nonperiodic functions $f:\mathbb{R}\rightarrow\mathbb{C}$ is well-defined for absolutely integrable functions as
\begin{displaymath}
\widehat{f}(k)=\frac{1}{\sqrt{2\pi}}\int_{-\infty}^{\infty}{f(x)e^{-\mathrm{i}kx}dx}, \ \ \ \ \ \forall \ k \in \mathbb{R},
\end{displaymath}
otherwise, the Fourier operator is extended to the space of tempered distributions $\mathcal{S}'$ as described in~\cite{MR1826392}. For integrable $2\pi$-periodic  functions we adopt the definition
\begin{displaymath}
\widehat{f}(k) = \frac{1}{2\pi} \int_{-\pi}^{\pi} f(x) e^{-\mathrm{i}kx} dx, \ \ \ \ \ \forall \ k \in \mathbb{Z},
\end{displaymath}
and the operator is extended to the space of periodic distributions $\mathcal{P}'$ as well.

Sobolev spaces for the nonperiodic domain case are defined as
\begin{displaymath}
H^s_{\mathbb{R}}=
\left\{f \in \mathcal{S}'(\mathbb{R});\ \widehat{f}\ \text{is measurable and}\ \int_{-\infty}^{\infty} \left( 1 + k^2 \right)^s \left|\widehat{f}(k)\right|^2 dk<\infty\right\},
\end{displaymath}
where $\mathcal{S}'(\mathbb{R})$ is the set of tempered distributions,  which are continuous linear functionals on the Schwartz space
\begin{displaymath}
\mathcal{S}(\RR)=\left\{f\in C^{\infty}(\mathbb{R});\    \sup_{x\in \R}\left|x^mf^{(n)}(x)\right|<\infty, \forall m,n \in \N\right\}.
\end{displaymath}
In $H^s_{\mathbb{R}}$ we define an inner product as
\begin{displaymath}
\left( f,g\right)_s=    \int_{-\infty}^{\infty} (1+k^2)^s \widehat{f}(k)\overline{\widehat{g}(k)}dk.
\end{displaymath}
In the periodic domain
\begin{displaymath}
H^s_{\text{per}} =
\left\{ f \in \mathcal{P}';\  \sum_{k=-\infty}^{\infty} (1+k^2)^s \left| \widehat{f}(k) \right|^2 < \infty\right\},
\end{displaymath}
where $\mathcal{P}'$ is the set of periodic distributions which are continuous linear functionals on
$\mathcal{P}=C_{\text{per}}^{\infty}(\mathbb{R})$, the set of $2\pi$-periodic infinitely smooth functions. 
The corresponding inner product is
\begin{displaymath}
\left( f,g \right)_s=    \sum_{k=-\infty}^{\infty} (1+k^2)^s \widehat{f}(k)\overline{\widehat{g}(k)}.
\end{displaymath}
For simplicity, Sobolev spaces are denoted hereafter by $H^s$ since the results proven in the sequel are valid for both nonperiodic and periodic cases.  Complete notation will be used only if it is necessary to point out some difference.

\subsection{Linearized equation results}   

The linearization of equation~(\ref{r-benjamin-family}) around the zero solution leads to the linear equation 
\begin{displaymath}
\phi_t + \phi_x -a \phi_{xxt} - b \mathcal{L}(\phi_{xt}) = 0,
\end{displaymath}
which  can be written in the frequency domain as 
\begin{equation}\label{cauchym2}
\widehat{\phi}_t(k,t)=-\mathrm{i}\varphi_j(k)\widehat{\phi}(k,t),
\end{equation}
where
\begin{displaymath}
\varphi_j(k)=\frac{k}{m_j(k)}
\end{displaymath}
and
\begin{displaymath}
m_j(k)=\left\{\begin{array}{l}
{1+b\left|k\right|+ak^2},\ \mbox{if} \ j=1,\\
{1+bk\coth(hk)+ak^2},\ \mbox{if} \ j=2,
\end{array}
\right.
\end{displaymath}
so that $j= 1$ corresponds to $\mathcal{L}=\mathcal{H}$ and $j=2$ to $\mathcal{L}=\mathcal{T}$. For $k=0$ the limit value in the expression for $m_2$ is considered. In the periodic case the domain of $m_j$ and $\varphi_j$ is restricted to $k\in \mathbb{Z}$.

We denote by $\varphi_j(D_x)$, where $D_x = -\mathrm{i} \partial_x$,  the linear operator such that 
\begin{displaymath}
\widehat{\varphi_j(D_x)\psi}(k) = \varphi_j(k)\widehat{\psi}(k),  \quad j=1, 2,  \, \, k \in \mathbb{R} \text{ or }  k \in \mathbb{Z}.
\end{displaymath}
We also write 
\begin{displaymath}
A_j\psi(x)=-\mathrm{i}\varphi_j(D_x)\psi(x), 
\end{displaymath}
where $\widehat{A_j\psi}(k)=-\mathrm{i}\varphi_j(k)\widehat{\psi}(k)$,  for $j=1, 2$. The linear operator $A_j: H^s \rightarrow H^s$ generates a unitary group denoted by $S_j(t)=e^{A_j t}$ such that $\widehat{S_j(t)}(k)=e^{-\mathrm{i}\varphi_j(k) t}$ and the linear problem associated to equation~(\ref{cauchym2}) with initial condition~$\eta_0\in H^s,\, s\in \mathbb{R}$, has a unique solution in $C^{1}(\mathbb{R}, H^s)$ of the form $\phi(t)=S_j(t)\eta_0$, as stated and proven in the following Theorem.

\begin{Theorem}\label{boacolocacaocauchylinear}
For $s\in \mathbb{R}$ and $\eta_0\in H^s$ the Cauchy problem
\begin{equation}\label{cauchylinear}
\left\{
\begin{array}{l}
\phi \in C\left(\mathbb{R},H^s\right), \\
\phi_t=A_j\phi, \\
\phi(0)=\eta_0\in H^s,
\end{array}
\right.
\end{equation}
has a unique solution $\phi\in C^1\left(\mathbb{R},H^s\right)$ of the form 
\begin{displaymath}
\phi(x,t)=S_j(t)\eta_0(x)=\left(\widehat{\eta_0}(k)e^{-\mathrm{i}\varphi_j(k)t}\right)^{\vee}(x),
\end{displaymath}
where the symbol $\vee$ denotes the inverse Fourier transform.
\end{Theorem}

\begin{proof} By definition $A_j$ is a linear operator in its domain $D(A_j)\subseteq H^s$.  Stone  Theorem and its reciprocal allow us to show that $A_j$ generates a unitary group. First, some bounds for $m_j$ and $\varphi_j$, $j=1,2$ are necessary. Since $\left|k\right|\leq k\coth(k)$, $\forall k\in \mathbb R$, then
\begin{displaymath}
m_2(k)\geq m_1(k)=1+b\left|k\right|+ak^2\geq (b+2\sqrt{a})\left|k\right|,\forall k \in \mathbb{R}.
\end{displaymath}
Therefore,
\begin{equation}\label{limitacao2}
\left|\varphi_2(k)\right|\leq \left|\varphi_1(k) \right|\leq \frac{1}{b+2\sqrt{a}}.
\end{equation}

The inequalities involved in equation~(\ref{limitacao2}) imply that $A_j$ is densely defined, in fact $D(A_j)=H^s$ and $A_j:H^s \rightarrow H^s$ is a bounded, continuous linear operator
for both $j=1,2$. This is so because if $f \in H^s$, then
\begin{displaymath}
	\int_{-\infty}^{\infty} (1+k^2)^s\left|\varphi_j(k)\right|^2\left|\widehat{f}(k)\right|^2dk \leq \frac{1}{\left(b+2\sqrt{a}\right)^2}\left\|f\right\|_s^2,
	\end{displaymath}
so $\left\|A_j f\right\|_s \leq \left\|f\right\|_s/\left(b+2\sqrt{a}\right)$. In addition, if 
$f,g\in D(A_j)=H^s$, then
\begin{displaymath}
\left( A_j f,g\right)_s=   \int_{-\infty}^{\infty} (1+k^2)^s (-\mathrm{i}\varphi_j(k))\widehat{f}(k)\overline{\widehat{g}(k)}dk=   \int_{-\infty}^{\infty} (1+k^2)^s \widehat{f}(k)\overline{(\mathrm{i}\varphi_j(k))\widehat{g}(k)}dk,
\end{displaymath}
which is exactly $\left( f,-Ag \right)_s$ and $A_j^\ast=-A_j$. 
Similar conclusions are valid in the periodic case. 
Therefore, $A_j$ is the infinitesimal generator of the strongly continuous one-parameter  unitary group $\left\{S_j(t)=e^{-\mathrm{i}t\varphi_j(D_x)}\right\}_{t\in \mathbb{R}}$, such that $\widehat{S_j(t)}(k)=e^{-\mathrm{i}\varphi_j(k)t}$.
From system~(\ref{cauchym2})
\begin{equation}\label{linear-freq-sol}
\widehat{\phi}(k,t)=\widehat{\eta_0}(k)e^{-\mathrm{i}\varphi_j(k)t}.
\end{equation}
Consequently, $\left|\widehat{\phi}(k,t)\right|=\left|\widehat{\eta_0}(k)\right|$, 
$\forall t \in \mathbb{R}$, $k \in \mathbb{R}$ in the nonperiodic case or $k \in \mathbb{Z}$ in the periodic case, $j=1, 2$ and $\left\|\phi(\cdot,t)\right\|_s=\left\|\eta_0\right\|_s, \forall t \in \mathbb{R}$.
Also, from equation~(\ref{linear-freq-sol}) we have that
\begin{displaymath}
\phi(x,t)=\left(\widehat{\eta_0}(k)e^{-\mathrm{i}\varphi_j(k)t}\right)^{\vee}(x)=S_j(t)\eta_0(x).
\end{displaymath}
Then,  the Cauchy problem~(\ref{cauchylinear}) has a unique solution in $C^1\left(\mathbb{R},H^s\right)$ of the form $\phi(x,t)=S_j(t)\eta_0(x)$.
\end{proof}

\section{Local well-posedness}\label{Localwellposedness}

In this section we prove local well-posedness for the Cauchy problem associated to equation~(\ref{r-benjamin-family}).  

We first recall the definition of well-posedness from~\cite{MR1826392}.

\begin{Definition}\label{def:WLP}
Let $X, Y$ be two Banach spaces and $F:[-T_0,T_0]\times Y \rightarrow X$ a continuous function. The Cauchy problem 
\begin{equation}\label{eq:gen}
\left\{
\begin{aligned}
&\eta_t = F\left(t,\eta(t)\right) \ \ \in X, \\
&\eta(0)=\phi \ \ \in Y,
\end{aligned}
\right.
\end{equation}
is locally well-posed in $Y$ if
\begin{itemize}
\item[(a)] there exist $T \in (0,T_0]$ and a function $\eta \in C([-T,T];Y)$ such that $\eta(0)=\phi$ and the differential  equation is satisfied in the following sense
\begin{equation}\label{quo:inc}
\lim_{h \rightarrow 0} \left\| \frac{\eta(t+h) - \eta(t)}{h} - F\left(t,\eta(t)\right) \right\|_X = 0,
\end{equation}
where the derivatives at $t=-T$ and $t=T$ are computed from the right and the left side, respectively,
\item[(b)] the problem~(\ref{eq:gen}) has at most one solution in $C\left([-T,T];Y\right)$,
\item[(c)] the map $\phi \longmapsto \eta$ is continuous, that is, given $(\phi_n)_{n \in \mathbb{N}} \subset Y$, such that $\phi_n \stackrel{Y}{\longrightarrow} \phi^{*}$ and $\eta^{*} \in C\left([-T^{*},T^{*}];Y\right)$ the solution for the initial condition $\phi^{*}$. Then, for  $n$ sufficiently large, solutions $\eta_n$ corresponding to $\phi_n$ can be defined on the interval $[-T^{*},T^{*}]$ and the following limit holds,
\begin{displaymath}
\lim_{n \rightarrow \infty} \sup_{[-T^{*},T^{*}]} \| \eta_n(t) - \eta^{*}(t) \|_Y = 0.
\end{displaymath}
\end{itemize}
\end{Definition}

\begin{Remark}
Note that this definition includes the persistence property: $\eta(t) \in Y$, $\forall \ t \in [-T,T]$.
\end{Remark}

\subsection{Main local well-posedness result}

We shall prove the following Theorem.
\begin{Theorem}\label{teoremalocalsemigrupo}
Let $s\geq 0$, $\eta_0\in H^s$, then there exists $T=T\left(s,\left\|\eta_0\right\|_s\right)>0$ such that the nonlinear Cauchy problem
\begin{equation}
 \label{transformadaequacao}
\left\{\begin{array}{l}
\eta\in C\left([-T,T],H^s\right),\\
\eta_t=A_j\left(\eta-\frac{3}{4}\alpha \eta^2\right),\\
\eta(0)=\eta_0,
\end{array}\right.
\end{equation}
is locally well-posed. 
\end{Theorem}

To prove this result we use the Banach Fixed Point Theorem to show that there exists a solution, which is local in time, to the integral equation arising from problem~(\ref{transformadaequacao}) and for which all items in Definition~\ref{def:WLP} hold.

\subsection{Auxiliary result and proof}

We now state and prove a Lemma which will be used to prove Theorem~\ref{teoremalocalsemigrupo}.  This Lemma  will be used again in Section~\ref{Globalwellposedness}.

\begin{Lemma}\label{lema_w1}
Let $0\leq s$ and $0 \leq r \leq s+1$. Then there exists a constant $C_{s,r}>0$ such that  for any $u\in H^s$ and $v\in H^r$, $A_j (uv)\in H^r$ and 
\begin{equation}\label{eq:ineq_w1}
\|A_j (uv)\|_r\leq C_{s,r} \|u\|_s \|v\|_r.
\end{equation}
\end{Lemma}

\begin{proof} Let us first  observe that $A_j$ is a bounded operator from $H^l$ to $H^{l+1}$ for any $l\in\mathbb{R}$. Indeed,  since
\begin{displaymath}
\left|\varphi_2(k)\right| \leq \left|\varphi_1(k)\right|  = \frac{|k|}{1+b\left|k\right|+ak^2} \leq \frac{|k|}{\min\{a,1\}(1+k^2)} \leq \frac{1}{\min\{a,1\} \sqrt{1+k^2}},
\end{displaymath}
it is valid that
\begin{displaymath}
\int_{-\infty}^{\infty} (1+k^2)^{l+1} |\varphi_j(k) \hat{f}(k)|^2 dk \leq\frac{1}{\min\{a^2,1\}} \int_{-\infty}^{\infty} (1+k^2)^{l} |\varphi_j(k) \hat{f}(k)|^2 dk,
\end{displaymath}
therefore, $A_j f \in H^{l+1}$ and $\left\|A_j f \right\|_{l+1}\leq \|f\|_l/{\min\{a,1\}}$, which is also valid in the periodic case.

In order to prove that $\|uv\|_{r-1} \leq C \|u\|_s \|v\|_r$ we apply Theorem~5.1 in page 286 in~\cite{MR4339668}, which guarantees the continuity of the pointwise multiplication in Sobolev spaces. In our case, we set $p = p_1 = p_2 = 2, s = r-1, s_1 = s, s_2 = r$. Therefore, if $s+r>0$, $p\leq\min\{s,r\}$  and $p<s+r-1/2$, then there is a constant $C'_{p,s,r}>0$ such that  $\|u v\|_p \leq C'_{p,s,r} \|u\|_s \|v\|_r$. Consequently, by setting $p=r-1$ the desired result follows for  $s, r > 0$ satisfying the conditions of the Lemma.
In the periodic case, the result remains valid.

We prove now the case where $s=r=0$. Under these conditions,  
\begin{align*}
   \int_{-\infty}^{\infty} \frac{1}{1+k^2} |\widehat{uv}(k)|^2 dk & =   \int_{-\infty}^{\infty} \frac{2\pi}{1+k^2} \left| \int_{-\infty}^{\infty} \hat{u}(\tilde{k})\hat{v}(k-\tilde{k}) d\tilde{k} \right|^2 dk \\ & \leq \|u\|_{0}^2 \|v\|_{0}^2    \int_{-\infty}^{\infty} \frac{2\pi }{1+k^2} dk,
\end{align*}
therefore, $\|u v\|_{-1}^2 \leq C \|u\|_{0}^2 \|v\|_{0}^2$. The same inequality is valid in the periodic case by the absolute convergence of the series with terms of the form $1/(1+k^2)$, $k \in \mathbb{Z}$.

In conclusion, taking $C_{s,r}=\max\left\{ C'_{r-1,s,r},C\right\}/\min\{a,1\}$, the inequality~\ref{eq:ineq_w1} holds.
\end{proof}

Throughout the text, we will denote constants arising from inequalities involving the norm of certain Sobolev spaces by the letter $C$ accompanied by subscripts. Secondary constants used within the scope of a proof or part of a proof will be denoted by $K_{0},K_{1},K_{2}$ and so on. 

\subsection{Proof of the main result (local case)} 

\begin{proof}{(of Theorem~\ref{teoremalocalsemigrupo})} Since the domain of the infinitesimal generator $A_j$ is $H^s$, in order to guarantee that the right-hand side of equation~(\ref{transformadaequacao}) belongs to $H^s$ if $\eta(t) \in H^s$ for $t\in I$, where $I$ is an interval such that $0 \in I$, we apply Lemma~\ref{lema_w1} for $r=s$ to assure that $A_j\eta^2(t) \in H^s$.

The local Cauchy problem in integral form reads
\begin{equation}\label{eqintegral}
\eta(t)=S_j(t)\eta_0-\frac{3\alpha}{4}   \int_0^t S_j(t-t')A_j\eta^2(t')dt', 
\end{equation}
for $j=1,2$. Locally, the existence of solution for the integral equation~(\ref{eqintegral}) can be established by Banach Fixed Point Theorem considering the normed space 
\begin{displaymath}
X_T^s=C\left([-T,T],H^s\right), \qquad \left\|\eta\right\|_{X_T^s}=   \sup_{t\in [-T,T]}\left\|\eta(t)\right\|_s.
\end{displaymath}
The solution of problem~(\ref{cauchylinear}) belongs to $C(\mathbb{R}, H^s)$ and $\left\|S_j(t)\eta_0\right\|_s=\left\|\eta_0\right\|_s$, therefore, $S_j(\cdot)\eta_0 \in X_T^s$ and
\begin{displaymath}
\left\|S_j(\cdot)\eta_0\right\|_{X_T^s}=   \sup_{t\in [-T,T]}\left\|S_j(t)\eta_0\right\|_s=\left\|\eta_0\right\|_s, \forall T>0.
\end{displaymath}

In order to prove local well-posedness, two inequalities are necessary. First, note that the second term in equation~(\ref{eqintegral}) is bounded since
\begin{align*}
\left\|   \int_0^t S_j(t-t')A_j\eta^2(t')dt'\right\|_s & \leq\left|   \int_0^t\left\|A_j\eta^2(t')\right\|_sdt'\right| \\ & \leq C_{s,s}\left|   \int_0^t\left\|\eta(t')\right\|_s^2dt'\right|\leq C_{s,s}\left\|\eta\right\|_{X_T^s}^2\left|t\right|.
\end{align*}
Therefore,
\begin{equation}
\label{desi1}
   \sup_{t\in[-T,T]}\left\|\frac{3\alpha}{4}   \int_0^t S_j(t-t')A_j\eta^2(t')dt'\right\|_s\leq \frac{3\alpha}{4}C_{s,s}\left\|\eta\right\|_{X_T^s}^2T.
\end{equation}
Analogously, for $\eta,\nu \in X_T^s$
\begin{equation}\label{desi2}
\sup_{t\in[-T,T]}\left\|\frac{3\alpha}{4}   \int_0^t S_j(t-t')A_j\left(\eta^2(t')-\nu^2(t')\right)dt'\right\|_s
\leq 
\frac{3\alpha}{4}C_{s,s}\left\|\eta-\nu\right\|_{X_T^s}\left\|\eta+\nu\right\|_{X_T^s}T.
\end{equation}

In order to prove item~(a) in definition~\ref{def:WLP} by Banach Fixed Point Theorem let us define

\begin{displaymath}
\Lambda=\Lambda(T,M)=\left\{v\in C\left([-T,T],H^s\right);\, d(v,0)\leq M\right\},
\end{displaymath}
which we endow with the  metric $d(v,u)=\left\|v-u\right\|_{X_T^s}$.  Since $\Lambda$ is a closed subset of the complete metric space $C\left([-T,T],H^s\right)$, as verified in~\cite{MR3805030}, $\Lambda$ is also a complete metric space for all $M>0$ and $T>0$. 

Consider $M=2\left\|\eta_0\right\|_s$ and 
\begin{displaymath}
\begin{array}{llll}
J:&\Lambda & \rightarrow & \Lambda\\
&v&\mapsto &Jv: [-T,T]\rightarrow H^s,
\end{array}
\end{displaymath}
defined by 
\begin{displaymath}
Jv(t)=S_j(t)\eta_0-\frac{3\alpha}{4}   \int_0^t S_j(t-t')A_jv^2(t')dt'.
\end{displaymath}
Note that $Jv$ is continuous since if $t,r\in [-T,T]$, then
\begin{displaymath}
\begin{array}{l}
Jv(t)-Jv(r)=\left(S_j(t)-S_j(r)\right)\eta_0\\
-\frac{3\alpha}{4}\left[\left(S_j(t-r)-I\right) \int_0^t S_j(r-\tau)A_j v^2(\tau)d\tau- \int_t^{r}S_j(r-\tau)A_j v^2(\tau)d\tau\right],
\end{array}
\end{displaymath}
which tends to $0$ in $H^s$ as $t\to r$ because the integrands do not depend on $t$ and $\left\{S_j(t)\right\}_{t\in \mathbb{R}}$ is a unitary group and consequently a strongly continuous group.

Let $T'>0$ be such that if $T\leq T'$ then $J$ is well defined. In fact, if $v\in \Lambda$, then
\begin{displaymath}
\left\|Jv(t)\right\|_s\leq \left\|S_j(t)\eta_0\right\|_s+\left\|\frac{3\alpha}{4}   \int_0^tS_j(t-t')A_j v^2(t')dt'\right\|_s,
\end{displaymath}
by inequality~(\ref{desi1}), there exists a constant $C_{s,s}>0$ such that 
\begin{align*}
\left\|Jv\right\|_{X_T^s}& \leq \left\|\eta_0\right\|_s+   \sup_{t\in [-T,T]}\left\|\frac{3\alpha}{4}   \int_0^t S_j(t-t')A_j v^2(t')dt'\right\|_s\\
&\leq\left\|\eta_0\right\|_s+\frac{3\alpha}{4}C_{s,s}\left\|v\right\|^2_{X_T^s}T
\leq \frac{M}{2}\left(1+\frac{3\alpha}{2}C_{s,s}MT\right).
\end{align*}
Fixing $T'=2/(3\alpha C_{s,s} M)$, then $\left\|Jv\right\|_{X_T^s}\leq M$ for all $T\leq T'$.  We conclude that $Jv\in \Lambda$ if $v\in \Lambda$. 

In addition, for all $T<T'$, $J:\Lambda \rightarrow \Lambda$ is a contraction. In fact, if $u, v \in \Lambda$, then inequality~(\ref{desi2}) provides a constant $C_{s,s}>0$ such that
\begin{align*}
\left\|Ju-Jv\right\|_{X_T^s}&=    \sup_{t\in [-T,T]}\left\|\frac{3\alpha}{4}   \int_0^t S_j(t-t')A_j\left(u^2(t')-v^2(t')\right)dt'\right\|_s\\
&\leq\frac{3\alpha}{4}C_{s,s}\left\|u-v\right\|_{X_T^s}\left\|u+v\right\|_{X_T^s}T
\leq \frac{3\alpha}{2}C_{s,s}M\left\|u-v\right\|_{X_T^s}T,
\end{align*}
then for all $T<T'=2/(3\alpha C_{s,s} M)$ we conclude that
\begin{displaymath}
\left\|Ju-Jv\right\|_{X_T^s}\leq \frac{3\alpha}{2}C_{s,s}M\frac{T}{T'}T'\left\|u-v\right\|_{X_T^s}=q_T\left\|u-v\right\|_{X_T^s},
\end{displaymath}
where $q_T=T/T'\in [0,1)$.
Considering the positive value $T=\frac{T'}{2}$, we have a well-defined contraction mapping $J:\Lambda \rightarrow \Lambda$. By Banach Fixed Point Theorem, there exists a unique fixed point $\eta \in \Lambda \subset C\left([-T,T],H^s\right)$, which is a solution of the integral problem~(\ref{eqintegral}) and consequently, a strong solution of the original problem as defined in~\cite{MR2233925},  page $125$. 

To conclude the proof of item~(a) in Definition~\ref{def:WLP}, note that in our problem $F(\eta(t))= A_j\left(\eta(t)-\frac{3}{4}\alpha \eta^2(t)\right)$ and the corresponding expression in the limit of equation~(\ref{quo:inc}) is
\begin{equation}\label{quotient}
\left\|\frac{S_j(h)-I}{h}\eta(t)-\frac{3\alpha}{4}\frac{1}{h} \int_t^{t+h}S_j(h+t-t')A_j\eta^2(t')dt'-A_j\eta(t)+\frac{3\alpha}{4}A_j\eta^2(t)\right\|_s,
\end{equation}
which tends to zero as $h$ tends to zero. In fact, $h^{-1}(S_j(h)-I)\eta(t)\stackrel{h\to 0}{\longrightarrow} A_j\eta(t)$ and by the Mean Value Theorem there exists $\theta_h\in (0,1)$ such that for $h>0$

\begin{align*} 
&\left\|\frac{1}{h} \int_t^{t+h}S_j(h+t-t')A_j\eta^2(t')dt'-A_j\eta^2(t)\right\|_s \\
&\leq \frac{1}{h} \int_t^{t+h} \left\|S_j(h+t-t')A_j\eta^2(t')-A_j\eta^2(t)\right\|_sdt'\\
&= \left\|S_j(h-\theta_h h)A_j\eta^2(t+\theta_h h)-A_j\eta^2(t)\right\|_s,
\end{align*} 
which tends to zero as $h$ tends to zero. The same limit exists for $h<0$.  Therefore,
\begin{align*}
&\lim_{h\to 0}\left\|\frac{S_j(h)-I}{h}\eta(t)-\frac{3\alpha}{4}\frac{1}{h}     \int_t^{t+h}S_j(h+t-t')A_j\eta^2(t')dt'\right.\\
&\left. -A_j\eta(t)+\frac{3\alpha}{4}A_j\eta^2(t)\right\|_s=0,
\end{align*}
for $s\geq 0$ as required.

In order to prove uniqueness as stated in item~(b) in Definition~\ref{def:WLP}, consider $\eta_i \in C\left([-T_i,T_i],H^s\right)$, for $s\geq 0$ and $i=1,2$ such that
\begin{displaymath}
\eta_i(t)=S_j(t)\eta_0^i-\frac{3\alpha}{4}   \int_0^tS_j(t-t')A_j\eta_i^2(t')dt',
\end{displaymath}
where $\eta_0^i=\eta_i(0)$ for $i=1,2$.
Without loss of generality, we assume that $T_1\leq T_2$. For any $t\in [-T_1,T_1]$, the inequality~(\ref{desi2}) 
provides a constant $C_{s,s}>0$ such that
\begin{displaymath}
\left\|\eta_1(t)-\eta_2(t)\right\|_s
\leq\left\|\eta_0^1-\eta_0^2\right\|_s+\frac{3\alpha C_{s,s}}{4}\left|   \int_0^t \left(K_0+\left\|\eta_0^1\right\|_s+\left\|\eta_0^2\right\|_s\right) \left\|\eta_1(t')-\eta_2(t')\right\|_sdt'\right|,
\end{displaymath}
where $K_0=2\mbox{max}\left\{\left\|\eta_1(t)-\eta_0^1\right\|_{X_s^{T_1}},\left\|\eta_2(t)-\eta_0^2\right\|_{X_s^{T_1}}\right\}$. Therefore,
\begin{displaymath}
\left\|\eta_1(t)-\eta_2(t)\right\|_s \leq \left\|\eta_0^1-\eta_0^2\right\|_s+K_1\left|   \int_0^t\left\|\eta_1(t')-\eta_2(t')\right\|_s dt' \right|,
\end{displaymath}
where $K_1=\frac{3\alpha}{4}C_{s,s}\left(K_0+\left\|\eta_0^1\right\|_s+\left\|\eta_0^2\right\|_s\right)$. By Gronwall~inequality
\begin{displaymath}
\left\|\eta_1(t)-\eta_2(t)\right\|_s \leq \left\|\eta_0^1-\eta_0^2\right\|_s\exp\left(K_1\left|   \int_0^tdt'\right|\right)=\left\|\eta_0^1-\eta_0^2\right\|_s\exp\left(K_1\left|t\right|\right).
\end{displaymath}
In particular, if $\eta_0^1=\eta_0^2$ and both solutions are defined in $[-T,T]$, then $\eta_1=\eta_2$ and there is at most one solution in $C([-T,T],H^s)$, $s\geq 0$.

Finally, in order to prove item~(c) in Definition~\ref{def:WLP}, consider  $\left(\eta_n^0\right)_{n\in \mathbb{N}} \subset H^s$ such that $\eta_n^0 \to \eta_0^*$ in $H^s$ and $\eta^* \in C([-T^*,T^*], H^s)$ is the solution with initial condition $\eta_0^*$, $T^* >0$. Let $\eta_n\in C\left(\left(-T_n,T_n\right), H^s\right)$ be the solution with initial condition $\eta_0^n$ and $\overline{T_n}=   \min\left\{T_n,T^*\right\}$, $n \in \mathbb{N}$, where $T_n>0$ is such that $\left(-T_n,T_n\right)$ is the maximal symmetric interval where the solution $\eta_n$ is defined. For $t\in \left(-\overline{T_n},\overline{T_n}\right)$, it is possible to show that
\begin{align*} 
&\left\|\eta_n(t)-\eta^*(t)\right\|_s\leq \left\|\eta_0^n-\eta_0^*\right\|_s+\\
&\left|\int_0^t\left(K_2\left\|\eta_n(t')-\eta^*(t')\right\|_s+K_3\left\|\eta_n(t')-\eta^*(t')\right\|_s^2\right)dt'\right|,
\end{align*}
where $K_2=3\alpha C_{s,s}M^*/2$, $K_3=3\alpha C_{s,s}/4$, $C_{s,s}>0$ is a constant provided by inequality~(\ref{desi2}) and $M^*=   \sup_{t\in [-T^*,T^*]} \left\|\eta^*(t)\right\|_s$. 

For $t\in \left(-\overline{T_n},\overline{T_n}\right)$, let us define
\begin{displaymath}
\psi_n(t)=\left\| \eta_0^n-\eta_0^*\right\|_s+\left|   \int_0^t\left(K_2\left\|\eta_n(t')-\eta^*(t')\right\|_s+K_3\left\|\eta_n(t')-\eta^*(t')\right\|_s^2\right)dt'\right|,
\end{displaymath}
in particular, $\psi_n(0)=\left\|\eta_0^n-\eta_0^*\right\|_s$. If $t> 0$ 
\begin{displaymath}
\psi_n'(t)= K_2\left\|\eta_n(t)-\eta^*(t)\right\|_s+K_3\left\|\eta_n(t)-\eta^*(t)\right\|_s^2 
\leq K_2\psi_n(t)+K_3\psi_n^2(t),
\end{displaymath}
consequently $\frac{K_2 \psi_n'(t)}{K_2\psi_n(t)+K_3\psi_n^2(t)}\leq K_2$. Integration from zero to $t$ leads to the  inequality
\begin{equation*}
\psi_n(t)\left[\frac{K_2+K_3\left\|\eta_0^n-\eta_0^*\right\|_s\left(1-e^{K_2 t}\right)}{K_2+K_3\left\|\eta_0^n-\eta_0^*\right\|_s}\right]\leq \frac{K_2\left\|\eta_0^n-\eta_0^*\right\|_se^{K_2 t }}{K_2+K_3\left\|\eta_0^n-\eta_0^*\right\|_s},
\end{equation*}
 if $K_2+K_3\left\|\eta_0^n-\eta_0^*\right\|_s\left(1-e^{K_2 t }\right) > 0$, which is the case when
\begin{displaymath}
t <T^*_n=\frac{1}{K_2}\ln\left(\frac{K_2+K_3\left\|\eta_0^n-\eta_0^*\right\|_s}{K_3\left\|\eta_0^n-\eta_0^*\right\|_s}\right).
\end{displaymath}

The same argument applies to the case $t<0$ and we get that 
\begin{equation}
\left\|\eta_n(t)-\eta^*(t)\right\|_s\leq \psi_n(t) \leq \frac{K_2\left\|\eta_0^n-\eta_0^*\right\|_se^{K_2\left|t\right|}}{K_2+K_3\left\|\eta_0^n-\eta_0^*\right\|_s\left(1-e^{K_2\left|t\right|}\right)}
\label{eqpsin}
\end{equation}
for all $t \in \left(-\overline{T_n},\overline{T_n}\right)\cap \left(-T_n^*,T_n^*\right)$.  Since $\left\|\eta_0^n-\eta_0^*\right\|_s \stackrel{n\rightarrow \infty}{\longrightarrow}0$, choosing $n$ sufficiently large guarantees that $T^*_n>T^*$ and inequality~(\ref{eqpsin}) is valid for all $t\in \left(-\overline{T_n},\overline{T_n}\right)$. In addition, if $T_n\leq T^*$, then $\overline{T_n}=T_n$ and 
\begin{displaymath}
\left\|\eta_n(t)\right\|_s\leq \left\|\eta_n(t)-\eta^*(t)\right\|_s+\left\|\eta^*(t)\right\|_s, \forall t\in \left(-T_n,T_n\right).
\end{displaymath}
This means that $\left\|\eta_n(\cdot,t)\right\|_s$ is bounded in $\left(-T_n,T_n\right)$, which contradicts the maximality of $T_n$. Therefore, $T_n>T^*$ for sufficiently large $n$, estimate~(\ref{eqpsin}) is valid $\forall t \in \left[-T^*,T^*\right]$ and
\begin{displaymath}
   \lim_{n\to \infty} \left\|\eta_n-\eta^*\right\|_{X_{T^*}^s}=   \lim_{n\to \infty}   \sup_{t\in [-T^*,T^*]} \left\|\eta_n(t)-\eta^*(t)\right\|_s=0.
\end{displaymath}
We conclude that the map $\eta_0 \longmapsto \eta$ is continuous and from definition~\ref{def:WLP}, problem~(\ref{transformadaequacao}) is locally well-posed for $s\geq 0$.

\end{proof}

\section{Global well-posedness}\label{Globalwellposedness}

In this section the global well-posedness of problem \eqref{transformadaequacao} is proven. We recall that a  Cauchy problem is globally well-posed if the Definition~\ref{def:WLP} is satisfied for any $T>0$ with $T_0=\infty$.

\subsection{Main results} We shall prove the following Theorem.

\begin{Theorem}\label{teoremaglobaltodos}
Let $s\geq 0$ and $\eta_0 \in H^s$, then the nonlinear Cauchy problem
\begin{equation}\label{eq:CP_global}
\left\{
\begin{aligned}
&\eta_t=A_j\left(\eta-\frac{3}{4}\alpha \eta^2\right),\\
&\eta(0)=\eta_0,
\end{aligned}
\right.
\end{equation}
is globally well-posed in $H^s$, in particular there is a unique $\eta \in C\left(\mathbb R; H^s\right)$ that solves this problem.
\end{Theorem}

The proof of the this Theorem will be divided into two steps. First, we address the case where $s\geq 1$ and in the sequence the case where $0\leq s <1$.  Since the techniques applied in each step are somehow different, in order to simplify the exposition we formulate the results corresponding to the first case  in the following Theorem.  

\begin{Theorem}
\label{teoremaglobal1}
Let $s \geq 1$ and $\eta_0 \in H^s$, then the nonlinear Cauchy problem 
\begin{displaymath}
\left\{
\begin{aligned}
&\eta_t=A_j\left(\eta-\frac{3}{4}\alpha \eta^2\right),\\
&\eta(0)=\eta_0,
\end{aligned}
\right.
\end{displaymath}
is globally well-posed in $H^s$.
\end{Theorem}

The proof of Theorem~\ref{teoremaglobal1} relies on the  combination of an extension principle and the existence of an appropriate conserved quantity for the local solutions (see Lemmas~\ref{lemaEP} and~\ref{lema4}). On the other hand, the proof of the remaining part of Theorem~\ref{teoremaglobaltodos} consists in the construction of solutions of the Cauchy problem using a decomposition of the initial datum in low and high-frequency components that takes advantage of the regularity of the former and the smallness of the later. Another basic ingredient of the proof is the well-posedness of a Cauchy problem related to the evolution and interaction of those components. This construction is  motivated by~ \cite{MR2461849}. The corresponding assertion is formulated in the following Theorem. 

\begin{Theorem}\label{localw}
Let $s\geq 0$, $ s+1\geq r\geq 0$,  $T_0>0$ and $u\in C\left([-T_0,T_0]; H^s\right)$. Then the nonlinear Cauchy problem
\begin{equation}\label{cauchywlocal}
\left\{\begin{aligned}
&w_t= A_j\left( w-\frac{3\alpha}{4}( u w + w^2) \right) \\
&w(0)=w_0 \in H^r,
\end{aligned}\right.
\end{equation}
is locally well-posed in $H^r$. Moreover, if $1\leq r\leq s+1$ the local solution can be extended to the whole interval $[-T_0,T_0]$, i.e. there is a unique solution $w\in C\left([-T_0,T_0]; H^{r}\right)$ of problem~\eqref{cauchywlocal}.
\end{Theorem}

\begin{Remark}
Theorem~\ref{teoremaglobal1} is a special case of  Theorem~\ref{localw}, where $T_0=\infty$ and $u(t)\equiv 0$. Nevertheless, we present Theorem~\ref{teoremaglobal1} as an independent result in order to simplify the exposition.
\end{Remark}

\subsection{Auxiliary results and proofs} 

In order to proof Theorems~\ref{teoremaglobal1}--~\ref{localw},   we need some auxiliary Lemmas.

\begin{Lemma}[Extension principle]\label{lemaEP}
Let $\eta \in C\left((T^-,T^+); H^s\right)$, with $s\geq 0$, be the maximal solution of the Cauchy problem described in Theorem~\ref{teoremalocalsemigrupo}. Then: 
\begin{enumerate}
\item Either $T^+ =+\infty$  or $T^+ \neq +\infty$  and  $\limsup_{t\uparrow T^+} \|\eta(t)\|_s = \infty$. 
\item Either $T^- =-\infty$  or $T^- \neq-\infty$ and $\limsup_{t\downarrow T^-} \|\eta(t)\|_s = \infty$. 
\end{enumerate} 
\end{Lemma}

\begin{proof}
We present the proof for the claim related to $T^+$, the other claim may be proven in much the same way.  

If $T^+ =+\infty$, there is nothing to prove.  Let us suppose that $T^+ \neq +\infty$. To obtain a contradiction, assume that $\limsup_{t\uparrow T^+} \|\eta(t)\|_s < \infty$. Hence there exists $M>0$ such that $\sup_{0\leq t < T^+} \|\eta(t)\|_s \leq M$.  Since $\eta(t)$ satisfies equation \eqref{eqintegral}, it follows that for any $t, \tilde{t}\in[0,T^+)$
\begin{displaymath}
\eta(\tilde{t}) -\eta(t) = \left(S_j(\tilde{t}-t)-I\right) \eta(t) -\frac{3\alpha}{4} S_j(\tilde{t}-t)\int_t^{\tilde{t}} S_j(t-t') A_j \eta^2(t') dt'.
\end{displaymath}
From the above and the inequality of Lemma~\ref{lema_w1}, we have
\begin{displaymath}
\|\eta(\tilde{t}) -\eta(t) \|_s \leq M \|S_j(\tilde{t}-t)-I\| + \frac{3\alpha}{4} C_{s,s} M^2 |\tilde{t}-t|.
\end{displaymath}
Consequently, $\|\eta(\tilde{t}) -\eta(t) \|_s$ converges to zero as $\tilde{t}-t$ goes to zero, and there exists $\eta^+\in H^s$ such that $\lim_{t\uparrow T^+}\eta(t) = \eta^+$. Note that $\eta^+$ is given by 
\begin{equation}\label{eq:eta_plus}
\eta^+  = S_j(T^+)\eta_0 -\frac{3\alpha}{4} \int_{0}^{T^+} S_j(T^+-t') A_j \eta^2(t') dt'.
\end{equation}
By the local well-posedness Theorem~\ref{teoremalocalsemigrupo}, there exist $T'>0$ and a local solution $\bar\eta\in C([T^+-T',T^++T']; H^s)$ such that $\bar\eta(T^+) = \eta^+$, satisfying the integral equation
\begin{equation}\label{eq:bar_eta}
\bar\eta(t) = S_j(t-T^+)\eta^+ -\frac{3\alpha}{4} \int_{T^+}^{t} S_j(t-t') A_j \bar\eta^2(t') dt'.
\end{equation}
Let us define the function $\tilde\eta\in C((T^-,T^++T']; H^s)$ as
\begin{displaymath}
\tilde\eta(t) = 
\begin{cases}
\eta(t),&\text{for $T^-<t<T^+$},\\
\bar\eta(t),&\text{for $T^+ \leq t\leq T^++T'$}.
\end{cases}
\end{displaymath}
Combining equations \eqref{eq:eta_plus} and \eqref{eq:bar_eta} we conclude that  $\tilde\eta$ is an extension of $\eta$. This contradicts the initial assumption,  thus  $\limsup_{t\uparrow T^+} \|\eta(t)\|_s = \infty$ which establishes our claim.    
\end{proof}

We introduce now a norm, equivalent to the $H^1$ norm,  which is key in our use of the extension principle.  

\begin{Lemma}\label{normasequivalentes}
Consider the norm 

\begin{displaymath}
\interleave f \interleave_1=\left[   \int_{-\infty}^{\infty}m_j(k)\left|\widehat{f}(k)\right|^2dk\right]^\frac{1}{2}.
\end{displaymath}
Then $\left\|\cdot\right\|_1$ and $\interleave \cdot \interleave_1$ are equivalent.
\end{Lemma}
\begin{proof}
Let us define the functions  $ \varrho_j(k) = \tfrac{m_j(k)}{k^2+1}$, ($j=1,2$) and observe that  $\lim_{|k|\to\infty} \varrho_j(k) = a>0$. Then there is $k_0>0$ such that for $|k| > k_0$
\begin{displaymath}
\frac{a}{2} < \varrho_j(k) < \frac{3a}{2}.
\end{displaymath}
Moreover, since the functions $ \varrho_j$ are continuous and positive in $\mathbb{R}$ there exist $\widetilde{K}_1$ and $\widetilde{K}_2$ such that 
\begin{displaymath}
0< \widetilde{K}_1 \leq \varrho_j(k) \leq \widetilde{K}_2
\end{displaymath}
for any $k\in[-k_0,k_0]$ and $j=1, 2$. Setting $K_1 = \min \left\{\widetilde{K}_1, \tfrac{a}{2}\right\}$ and $K_2 = \max \left\{ \widetilde{K}_2, \tfrac{3a}{2} \right\}$, we have $K_1(1+k^2) \leq m_j(k)\leq K_2(1+k^2)$ for any $k\in\mathbb{R}$ ($j=1,2$) and the claim follows.
\end{proof} 

In the  next Lemma we show that the norm introduced in Lemma~\ref{normasequivalentes} is a conserved quantity.

\begin{Lemma}
\label{lema4}
If $\eta \in C\left([-T,T],H^s\right)$ with $s\geq 1$ satisfies the Cauchy problem described in Theorem~\ref{teoremalocalsemigrupo}, then for all $t\in [-T,T]$ 
\begin{displaymath}
\interleave \eta(t) \interleave_1=\interleave \eta_0 \interleave_1,
\end{displaymath}
where  
$\interleave f \interleave_1=\left[   \int_{-\infty}^{\infty}m(k)\left|\widehat{f}(k)\right|^2dk\right]^\frac{1}{2}$ with $m(k)=m_1(k)$ if $\mathcal{L}=\mathcal{H}$ or $m(k)=m_2(k)$ if $\mathcal{L}=\mathcal{T}$ is considered.
\end{Lemma}

\begin{Remark}
We have a similar result for the periodic case by considering 

\begin{displaymath}
\interleave f \interleave_1=\left[   \sum_{k=-\infty}^{\infty}m(k)\left|\widehat{f}(k)\right|^2\right]^\frac{1}{2}. 
\end{displaymath}
\end{Remark}
\begin{proof}{(of Lemma~\ref{lema4})}
Since $\eta$ satisfies the Cauchy problem described in Theorem~\ref{teoremalocalsemigrupo} in the distributional sense for $\eta \in C\left([-T,T],H^s\right)$, by means of Lemma~\ref{lema_w1} when $r=s$, 
it follows that
\begin{displaymath}
\eta_t=A_j\left(\eta-\frac{3}{4}\alpha\eta^2\right) \in C\left(\left[-T,T\right],H^s\right).
\end{displaymath}
Therefore, $\eta \in C^1\left(\left[-T,T\right],H^s\right)$. Consider the sequence
$\left(\eta_n\right)_{n\in\mathbb N} \subset C^1\left(\left[-T,T\right],H^\infty\right)$,
converging to the solution $\eta\in C^1\left(\left[-T,T\right],H^s\right)$. 
Define the nonlinear operator 
\begin{displaymath}
G: C^1\left(\left[-T,T\right],H^s\right) \to C\left(\left[-T,T\right],H^{s}\right),
\end{displaymath}
such that $G(v)=v_t-A_j\left(v-\tfrac{3\alpha}{4}v^2\right)$. Operator $G$ is continuous,  therefore
\begin{equation}
\lim_{n \to \infty} G\left(\eta_n\right)=G(\eta)=0,
\label{eq11}
\end{equation}
in $C\left(\left[-T,T\right],H^{s}\right)$. 

Set $w_n = \left( m(\cdot) \overline{\hat{\eta}_n}\right)^\vee$, then $w_n\in C^1\left(\left[-T,T\right],H^{s-2}\right)$.  Since $s\geq 1$, then $s-2\geq -1$, therefore we have the continuous embeddings $H^s \hookrightarrow H^1$ and $H^{s-2} \hookrightarrow H^{-1}$.  Since $H^{-1}$ is the dual of $H^1$, the limit provided by equation~\eqref{eq11} when $n \to \infty$ implies the decay of the following duality bracket:
\begin{equation}
\left|\left\langle w_n(t), G\left(\eta_n(t)\right)\right\rangle_{1}\right| \leq \left\|w_n\right\|_{-1} \left\|G\left(\eta_n\right)\right\|_{1}\leq \left\|w_n\right\|_{s-2}\left\|G\left(\eta_n\right)\right\|_{s} \to 0.
\label{eq22}
\end{equation}

From the real part (denoted by $\Re$), of $\left\langle w_n(t), G\left(\eta_n(t)\right)\right\rangle_{1}$, applying Parseval identity and using that $m$ is a real-valued function,  we get
\begin{equation}
\begin{aligned}\label{eq:re_w_eta_t}
\Re\left\langle w_n,  \partial_t\eta_{n} \right\rangle_1 & = \frac{1}{2}\left[\left\langle w_n,  \partial_t\eta_{n} \right\rangle_1 + \overline{\left\langle w_n,  \partial_t\eta_{n} \right\rangle_1}\right] \\
&=\frac{1}{2}\int_{-\infty}^{\infty} \left[ m(k) \overline{\widehat{\eta_n}}(k) \partial_t\widehat{\eta_n}(k) + \overline{m(k) \overline{\widehat{\eta_n}}(k) \partial_t\widehat{\eta_n}(k)}  \right] dk \\
&=\frac{1}{2}\int_{-\infty}^{\infty}  m(k) \left[ \overline{\widehat{\eta_n}}(k) \partial_t\widehat{\eta_n}(k) +  \widehat{\eta_n}(k) \partial_t\overline{\widehat{\eta_n}}(k) \right] dk \\
&=\frac{1}{2}\int_{-\infty}^{\infty}  m(k) \partial_t \left| \widehat{\eta_n}(k) \right|^2 dk \\
&=\frac{1}{2}\frac{d}{dt}  \interleave \eta_n \interleave_1^2.
\end{aligned}
\end{equation}
Also, by Parseval identity

\begin{equation}
\begin{aligned}\label{eq:w_Aj}
\left\langle w_n, A_j\left( \eta_n  -\frac{3\alpha}{4} \eta_n^2\right)\right\rangle_1 & = \int_{-\infty}^{\infty} m(k) \overline{\widehat{\eta_n}}(k) \left[-\mathbf{i}\varphi_j(k)\widehat{\left(\eta_n -\frac{3\alpha}{4} \eta_n^2\right)} (k)\right] dk \\
& = -\int_{-\infty}^{\infty} \overline{\widehat{\eta_n}}(k) \left[\mathbf{i}k \widehat{\left(\eta_n -\frac{3\alpha}{4} \eta_n^2\right)} (k)\right] dk \\
& = -\int_{-\infty}^{\infty}\eta_n(x)\partial_x\left( \eta_n(x) -\frac{3\alpha}{4} \eta_n^2(x)\right) dx \\
& =  -\int_{-\infty}^{\infty}\partial_x\left(\frac{1}{2}\eta^2_n(x) -\frac{\alpha}{4} \eta_n^3(x) \right)dx = 0.
\end{aligned}
\end{equation}
By integrating $\langle w_n, G\left(\eta_n\right)\rangle_1$  from $0$ to $t$ (with $t\in \left[-T,T\right]$) and using equations \eqref{eq:re_w_eta_t} and \eqref{eq:w_Aj}, we get
\begin{equation}
2\Re \int_0^t\left\langle w_n(t'), G\left(\eta_n(t')\right)\right\rangle_1 dt'  = \interleave \eta_n(t) \interleave_1^2-\interleave \eta_n(0) \interleave_1^2.
\label{eq33}
\end{equation}
The same results are valid in the periodic case. Finally,  taking limit in \eqref{eq33} as $n\to\infty$ and using equation \eqref{eq22} it follows that $\interleave \eta(t) \interleave_1=\interleave \eta_0\interleave_1$.
\end{proof}

We have now obtained all the necessary results to prove global existence,  i.e.  Theorem~\ref{teoremaglobal1}. 

\subsection{Proof of the main result} 

\begin{proof}{(of Theorem~\ref{teoremaglobal1})} Take $s\geq 1$ and let $\eta\in C((T^-,T^+); H^s)$ represents the maximal solution of the Cauchy problem \eqref{eq:CP_global} in $H^s$, where $T^{\pm} = T^{\pm}(\eta_0)$. We claim that  $T^{\pm} = \pm\infty$. 

Since the solution satisfies the integral equation

\begin{displaymath}
\eta(t)=S_j(t)\eta_0-\frac{3\alpha}{4}   \int_0^t S_j(t-t')A_j\eta^2(t')dt',
\end{displaymath}
for any $t\in(T^-,T^+)$, using that $A_j$ is a linear bounded operator (see inequality~\eqref{limitacao2}) it follows that
\begin{equation}
\begin{aligned}\label{eq:ineq_norm}
\left\|\eta(t)\right\|_s &\leq \left\|\eta_0\right\|_s+\frac{3\alpha}{4}\left|   \int_0^t \left\|A_j\eta^2(t')\right\|_sdt'\right| \\
& \leq \left\|\eta_0\right\|_s+ \frac{3\alpha}{4(b+2\sqrt{a})} \left| \int_0^t \left\|\eta^2(t')\right\|_s dt'\right|.
\end{aligned}
\end{equation}
As a consequence of Lemma~\ref{lema4} and Sobolev Lemma (see for instance~\cite{MR1826392}), there is a constant $C_{S}$ such that  
\begin{equation} \label{eq:ineq_inf}
\|\eta(t)\|_\infty\leq C_{S} \, \|\eta_0\|_1
\end{equation}  
for any $t\in(T^-,T^+)$. Moreover, using  inequalities \eqref{eq:ineq_norm}, \eqref{eq:ineq_inf}, and recalling that there is a constant $K_0$ such that $\|v^2\|_s\leq K_0\|v\|_\infty \|v\|_s$ for any $v\in L^\infty\cup H^s$ (see for instance~\cite{MR2121254}), we obtain

\begin{align*}
\left\|\eta(t)\right\|_s &  \leq \left\|\eta_0\right\|_s+ \frac{3\alpha K_0}{4(b+2\sqrt{a})} \left| \int_0^t \left\|\eta(t')\right\|_\infty \left\|\eta(t')\right\|_s dt'\right|\\
& \leq \left\|\eta_0\right\|_s+ \frac{3\alpha C_{S} K_0 \left\|\eta_0\right\|_1}{4(b+2\sqrt{a})} \left| \int_0^t  \left\|\eta(t')\right\|_s dt'\right|.
\end{align*}
From this, using Gronwall Lemma, we deduce that for any $t\in(T^-,T^+)$
\begin{displaymath}
\left\|\eta(t)\right\|_s \leq \|\eta_0 \|_s e^{K_1 |t|},
\end{displaymath}
where $K_1 = \tfrac{3\alpha C_{S} K_0 }{4(b+2\sqrt{a})} \| \eta_0 \|_1$. 

Finally, using Lemma~\ref{lemaEP} we can assert that $T^\pm = \pm\infty$. Indeed, if $T^+$ were finite, we would have $\limsup_{t\to T^+} \left\|\eta(t)\right\|_s  \leq \|\eta_0 \|_s e^{K_1 T^+}<\infty$ which contradicts the conclusions of that Lemma. Similar considerations apply to $T^-$.
\end{proof}

We proceed to prove the global well-posedness of the Cauchy problem \eqref{eq:CP_global} for the case where $0\leq s < 1$.

\begin{proof}{(of Theorem~\ref{teoremaglobaltodos})} Let $\eta_0\in H^s$ and fix an arbitrary $T>0$. For $N>0$ define 
\begin{displaymath}
\eta_N^0(x) =  \frac{1}{\sqrt{2\pi}}   \int_{\left|k\right|\geq N}e^{\mathrm{i} kx}\widehat{\eta_0}(k)dk
\end{displaymath}
in the nonperiodic case or
\begin{displaymath}
\eta_N^0(x) =  \sum_{\left|k\right|> N}e^{\mathrm{i} kx}\widehat{\eta_0}(k)
\end{displaymath}
in the periodic case.
Observe that $\eta_N^0\in H^s$ and $\lim_{N\to\infty} \|\eta_N^0\|_s = 0$. Hence, from Theorem~\ref{teoremalocalsemigrupo},  we can take $N$ sufficiently large such that the maximal solution in $H^s$ of the Cauchy problem
\begin{equation}\label{eq:CP_global_v}
\left\{
\begin{aligned}
&v_t=A_j\left(v-\frac{3}{4}\alpha v^2\right),\\
&v(0)=\eta_N^0,
\end{aligned}
\right.
\end{equation}
 is defined for any $t\in[-T,T]$. Let $v\in C([-T,T]; H^s)$ represents this solution.

Consider the Cauchy problem
\begin{equation}\label{eq:CP_aux}
\left\{
\begin{aligned}
& w_t=A_j\left(w-\frac{3\alpha}{4}( 2 v w + w^2)\right),\\
& w(0)=w_0,
\end{aligned}
\right.
\end{equation}
where $w_0=\eta_0-\eta_N^0$. Since its expression is given by
\begin{displaymath}
w_0(x) = \frac{1}{\sqrt{2\pi}}   \int_{-N}^N e^{\mathrm{i} kx}\widehat{\eta_0}(k)dk,
\end{displaymath}
or
\begin{displaymath}
w_0(x) = \sum_{k=-N}^N e^{\mathrm{i} kx}\widehat{\eta_0}(k)
\end{displaymath}
in the periodic case, it follows that $w_0\in H^r$ for any real $r$. In particular, taking $r=1$ and $u=2v$, Theorem~\ref{localw} guarantees the existence and uniqueness of a solution $w\in C([-T,T]; H^1)$. Since $s<1$, the Sobolev embedding ensures that $w\in C([-T,T]; H^s)$ and also shows that $w$ is a solution of the Cauchy problem~\eqref{eq:CP_aux} in $H^s$.

From what have already been proven, we conclude that $\tilde\eta = v+w\in C([-T,T]; H^s)$ solves the Cauchy problem \eqref{eq:CP_global}. Indeed, we have that $\tilde\eta(0) = v_0 + w_0 = \eta_0$, and  combining \eqref{eq:CP_global_v} and \eqref{eq:CP_aux} yields
\[
\begin{aligned}
\tilde\eta_t &= v_t + w_t = A_j\left(v-\frac{3\alpha}{4} v^2\right) + A_j\left(w-\frac{3\alpha}{4}( 2 v w + w^2)\right) \\
& = A_j\left(v+w-\frac{3\alpha}{4}(v+w)^2\right)\\
& = A_j\left(\tilde\eta-\frac{3\alpha}{4}\tilde\eta^2\right).
\end{aligned}
\]
Hence $\tilde\eta\in C([-T,T]; H^s)$ is a solution of  problem~\eqref{eq:CP_global}. 

Finally, since $T$ is arbitrary the solution of  problem~\eqref{eq:CP_global} is globally defined and  the proof is complete.
\end{proof}

We proceed to prove Theorem~\ref{localw}. The basic ideas of the proofs of the first and second claims of this Theorem are analogous to those used in the proof of Theorems~\ref{teoremalocalsemigrupo} and~\ref{teoremaglobal1}, respectively. Therefore, we highlight only the main steps of the proof.

\begin{proof}{(of Theorem~\ref{localw})} In order to prove the existence of solutions, we consider the following integral form of equation~\eqref{cauchywlocal}
\begin{equation}
\label{eqintegralw}
w(t)=S_j(t)w_0-\frac{3\alpha}{4}   \int_0^t S_j (t-t')A_j\left((uw)(t') +w^2(t')\right)dt',
\end{equation}
and apply the Banach Fixed Point Theorem {again} in a complete metric space 
\begin{displaymath}
\Lambda=\Lambda(T,M)=\left\{v\in C\left([-T,T],H^r\right); \left\|v\right\|_{X_T^{r}}\leq M\right\},
\end{displaymath}
with metric $d(v_1, v_2) = \left\|v_1-v_2\right\|_{X_T^{r}}$, where $\left\|v\right\|_{X_T^{r}}= \sup_{t\in [-T,T]}\left\|v(t)\right\|_{r}$. The values of $M$ and $T$ will be chosen conveniently in the next steps.

Given the value of $T_0$ in the statement of this Theorem, we consider $T\leq T_0$ and the map $J$ on $\Lambda$  defined as
\begin{displaymath}
(Jv)(t) = S_j(t)w_0-\frac{3\alpha}{4}   \int_0^t S_j (t-t')A_j\left((uv)(t') +v^2(t')\right)dt',\quad t\in[-T, T].
\end{displaymath}
By Lemma~\ref{lema_w1} and the properties of $A_j$ and $S_j$, we have $(Jv)(t)\in H^r$ for each $ t\in[-T, T]$. Moreover, $J:\Lambda\to X_T^r$. Indeed, for $t_1,t_2\in[-T,T]$ we get
\begin{multline*}
(Jv)(t_2)-(Jv)(t_1) = \left(S_j(\Delta t)-I\right)S_j(t_1) w_0 \\
-\frac{3\alpha}{4}\left[\left(S_j(\Delta t)-I\right) \int_0^{t_1} S_j (t_1-t')A_j\left((uv)(t') +v^2(t')\right)dt' \right. \\
\left. +  \int_{t_1}^{t_2} S_j (t_2-t')A_j\left((uv)(t') +v^2(t')\right)dt'  \right],
\end{multline*}
where $\Delta t = t_2-t_1$. Combining the strong continuity of the group $S_j$, the boundedness of $u$ and $v$, and Lemma~\ref{lema_w1} we get the continuity of $Jv$.

Set $M=\max\left\{2 \left\|w_0\right\|_{r},\left\|u\right\|_{X_{T_0}^s}\right\}$ in the definition of the space $\Lambda=\Lambda(T,M)$. Then, from equation~\eqref{eqintegralw} and by Lemma~\ref{lema_w1} we get 
\begin{align*}
\|Jv\|_{X_T^r} & \leq \left\|w_0\right\|_{r} + \frac{3\alpha}{4}T  \left( C_{r,s}\|u\|_{X_{T_0}^s} +C_{r,r}\|v\|_{X_T^r} \right) \|v\|_{X_T^r} \\
& \leq \frac{M}{2}\left(1 + \frac{3\alpha}{2} \widetilde{C}_{r,s} T M\right),
\end{align*}
where $\widetilde{C}_{r,s} = C_{r,s} + C_{r,r}$. Thus, taking $T\leq T' = \min\left\{ T_0, \frac{2}{3\alpha \widetilde{C}_{r,s} M}  \right\}$ we have $J(\Lambda)\subset \Lambda$.

Take $v_1, v_2\in\Lambda$, then 
\begin{multline*}
(Jv_2)(t)-(Jv_1)(t) =  -\frac{3\alpha}{4} \int_0^{t} S_j (t-t')A_j \left[u(t') \left(v_2(t')- v_1(t')\right) \right.  \\
 + \left. \left(v_2(t') + v_1(t')\right) \left(v_2(t')- v_1(t')\right)  \right] dt',
\end{multline*}
and the inequalities from Lemma~\ref{lema_w1} yield 
\begin{align*}
d(Jv_1, Jv_2) &\leq \frac{3\alpha}{4} T \left( C_{r,s}\|u\|_{X_{T_0}^s} + 2C_{r,r} \max\{\|v_1\|_{X_T^r}, \|v_2\|_{X_T^r}\} \right) \|v_1 - v_2\|_{X_T^r}\\
& \leq \frac{3\alpha}{2} \widetilde{C}_{r,s} T M d(v_1, v_2).
\end{align*}
Hence, choosing $T < T'$ it follows that $J$ is a contraction map on $\Lambda$. We conclude that there exists a solution $w\in C([-T,T]; H^r)$ of the integral equation~\eqref{eqintegralw}.

We proceed to show that $w_t  = A_j\left( w -\tfrac{3\alpha}{4}(uw+w^2)\right)$ in $H^r$. Fix $t \in [-T,T]$ and take $h$ such that $t+h \in [-T,T]$, then
\begin{multline}\label{eq:lim-w_t}
\frac{w(t+h)-w(t)}{h} = \left(\frac{S_j(h)-I}{h}\right) w(t) \\
-\frac{3\alpha}{4h}  \int_{t}^{t+h} S_j (t+h-t') A_j\left((uw)(t') +w^2(t')\right) dt'.
\end{multline}
From the properties of $S_j$ we see that the  first term in the right-hand side of \eqref{eq:lim-w_t} converges to $A_j w(t)$ as $h\to 0$. Additionally, considering $h>0$ and using the Mean Value Theorem there exists $\theta_h\in (0,1)$ such that 
\begin{align*} 
&\left\|\frac{1}{h}   \int_t^{t+h}S_j(t+h-t' )A_j \left((uw)(t') +w^2(t')\right) dt'-A_j\left((uw)(t) +w^2(t)\right) \right\|_s \\
&\leq \frac{1}{h}   \int_t^{t+h} \left\|S_j(t+h-t')A_j  \left((uw)(t') +w^2(t')\right) - A_j  \left((uw)(t) +w^2(t)\right)\right\|_s dt'\\
&= \left\|S_j(h-\theta_h h)A_j \left((uw)(t+\theta_h h) +w^2(t+\theta_h h)\right) -A_j \left((uw)(t) +w^2(t)\right) \right\|_s,
\end{align*} 
which tends to zero as $h$ goes to zero. The same reasoning applies to the case $h<0$. Consequently, the  second term in the right-hand side of \eqref{eq:lim-w_t} converges to $-\frac{3\alpha}{4h}  A_j  \left((uw)(t) +w^2(t)\right) $ as $h\to 0$, which establishes the desired conclusion.

Our next goal is to prove uniqueness. Let $w_i\in C([-T_i,T_i]; H^r)$ ($i=1,2$) with $0<T_i\leq T_0$ be two solutions of problem~\eqref{cauchywlocal}. Without loss of generality, we assume that $T_1\leq T_2$. From equation~\eqref{eqintegralw},  for any $t\in [-T_1,T_1]$ we have
\begin{multline*}
w_2(t)-w_1(t) =  S_j(t)\left(w_2(0)-w_1(0)\right) -\frac{3\alpha}{4} \int_0^{t} S_j (t-t')A_j \left[u(t') \left(w_2(t')- w_1(t')\right) \right.  \\
 + \left. \left(w_2(t') + w_1(t')\right) \left(w_2(t')- w_1(t')\right)  \right] dt'.
\end{multline*}
Hence, using Lemma~\ref{lema_w1}, it follows that
\begin{displaymath}
\|w_2(t)-w_1(t)\|_r \leq  \|w_2(0)-w_1(0)\|_r + \frac{3\alpha}{2} \widetilde{C}_{r,s} K_0 \left| \int_0^{t} \|w_2(t')-w_1(t')\|_r dt' \right|,
\end{displaymath}
where $K_0 = \sup_{t'\in [-T_1,T_1]} \max\left\{\|u(t')\|_s, \|w_1(t')\|_r, \|w_2(t')\|_r \right\}$, and Gronwall inequality implies that $w_2(t)-w_1(t) = 0$ for any $t\in [-T_1,T_1]$. 

Now, we turn to the proof of the continuous dependence of solutions. Let $w^* \in C([-T^*,T^*], H^{r})$ be a solution with initial condition $w_0^*$, where $0< T^* \leq T_0$. Consider a sequence $\left(w_n^0\right)_{n\in \N} \subset H^{r}$ satisfying $w_n^0 \to w_0^*$ in $H^{r}$ and let $w_n\in C\left([-T_n,T_n], H^{r}\right)$ be a solution with initial condition $w_0^n$, where $0<  T_n\leq T_0$, for all $n \in \N$. 

Consider $0<T<T^*$ and set $\overline{T_n}=\min\left\{T_n, T\right\}$, $n \in \N$. Then, from Lemma~\ref{lema_w1}, for any $t\in [-\overline{T_n},\overline{T_n}]$ we get
\begin{align*}
\left\|w_n(t)-w^*(t)\right\|_{r}  & \leq \left\| w_0^n- w_0^*\right\|_{r} \\ 
& + \left|\int_0^{t}  \left[K_1\|w_n(t') - w^*(t')\|_r  + K_2 \|w_n(t') - w^*(t')\|_r^2   \right] dt' \right|,
\end{align*}
where $K_1 = \tfrac{3\alpha}{4}\left(C_{r,s} \sup_{t'\in[-T_0,T_0]} \|u(t')\|_s +2 C_{r,r} \sup_{t'\in[-T^*,T^*]} \|w^*(t')\|_r\right)$ and $K_2 = \tfrac{3\alpha}{4}C_{r,r} \sup_{t'\in[-T^*,T^*]} \|w^*(t')\|_r$. Then, applying a  comparison argument similar to that in the proof of Theorem~\ref{teoremalocalsemigrupo}, it follows that 
\begin{equation}
\left\|w_n(t)-w^*(t)\right\|_{r}\leq  \frac{K_1\left\|w_0^n-w_0^*\right\|_{r}e^{K_1\left|t\right|}}{K_1+K_2\left\|w_0^n-w_0^*\right\|_{r}\left(1-e^{K_1\left|t\right|}\right)},
\label{eqpsinw}
\end{equation}
for any $t\in  [-\overline{T_n},\overline{T_n}]\cap (-T_n^*, T_n^*)$, where $T_n^* = \frac{1}{K_1}\ln\left(\frac{K_1 +K_2 \left\| w_0^n- w_0^*\right\|_{r} }{K_2 \left\| w_0^n- w_0^*\right\|_{r}}\right)$. Since $w_0^n\to w_0^*$, we have $T_n^*\to\infty$  and \eqref{eqpsinw} is valid for $t\in  [-\overline{T_n},\overline{T_n}]$ when  $n$ is sufficiently large. Furthermore, this implies that $w_n(t)$ remains bounded for $t\in[-\overline{T_n},\overline{T_n}]$ and proceeding as in the proof of Lemma~\ref{lemaEP} we conclude that the solution can be extended to the interval $[-T, T]$. Then, from \eqref{eqpsinw} it is easy to check that
$$\lim_{n\to \infty}   \sup_{t\in [-T, T]} \left\|w_n(\cdot,t)-w^*(\cdot,t)\right\|_{r}=0,$$
which concludes the proof of the continuous dependence.

Our next goal is to establish that if $r\geq 1$ the solution is defined in the whole interval $[-T_0, T_0]$. Consider a solution  $w\in C([-T_*, T_*], H^r)$ of problem \eqref{cauchywlocal} with $T_* <T_0$. Proceeding as in the proof of Lemma~\ref{lema4}, it follows that 
\begin{displaymath}
\interleave w(t) \interleave_1^2 = \interleave w(0) \interleave_1^{2} - \frac{3\alpha}{2}\int_0^t ( u(t'), w(t')\partial_x w(t') )_0 dt',
\end{displaymath}
then
\begin{align*}
\interleave w(t) \interleave_1^2 &\leq \interleave w_0 \interleave_1^2 +  \frac{3\alpha}{2}\left |\int_0^t \| u(t')\|_0 \|w(t')\|_\infty \|\partial_x w(t')\|_0  dt' \right | \\
&\leq \interleave w_0 \interleave_1^2 + K_3 \left |\int_0^t \|w(t')\|_1^2  dt' \right |,
\end{align*}
where $K_3 =  \tfrac{3\alpha}{2} C_{S} \sup_{t'\in[-T_0,T_0]} \|u(t')\|_s$ with the constant $C_{S}$ appearing in Sobolev Lemma (see equation~\eqref{eq:ineq_inf}).  As a consequence, applying Gronwall Lemma
\begin{displaymath}
\|w(t)\|_1 \leq K_4
\end{displaymath}
for any $t\in [-T_*, T_*]$, where $K_4$ is a constant independent of $T_*$.  From the integral formulation \eqref{eqintegralw} we have for any $t\in [-T_*, T_*]$
\begin{displaymath}
\|w(t)\|_r \leq \|w_0\|_r + K_5 \left| \int_0^t \|w(t')\|_r dt'  \right|,
\end{displaymath}
where
\begin{displaymath}
K_5 = \tfrac{3\alpha}{4} \left(C_{r,s}\sup_{t'\in[-T_0,T_0]} \|u(t')\|_s + \tfrac{1}{a+\sqrt{b}} C_{S} K_4 \right).
\end{displaymath}
Applying Gronwall Lemma, it follows that  $\|w(t)\|_r \leq \|w_0\|_r e^{K_5 T_0}$, for any $t\in [-T_*, T_*]$.  Hence, using an argument similar to that in the proof of Lemma~\ref{lemaEP}, this local solution can always be extended to a larger interval unless $T_* = T_0$. As a consequence the maximal solution is defined on the  interval  $[-T_0, T_0]$.
\end{proof}


\end{document}